\documentclass{birkau}

\usepackage{amsmath,amssymb,url}
\usepackage [all]{xy}
\usepackage{mathrsfs}
\usepackage{mathtools}
\usepackage{stmaryrd}
\usepackage{lmodern}

\usepackage{graphicx}
\usepackage{tikz}
\usepackage[table]{xcolor}
\usepackage{booktabs}

\numberwithin{equation}{section}

\theoremstyle{plain}
\newtheorem{theorem}{Theorem}[section]
\newtheorem{lemma}[theorem]{Lemma}
\newtheorem{proposition}[theorem]{Proposition}
\newtheorem{corollary}[theorem]{Corollary}

\theoremstyle{definition}
\newtheorem{definition}[theorem]{Definition}

\newtheorem{example}[theorem]{Example}

\DeclareMathOperator{\Aut}{Aut}
\DeclareMathOperator{\PSpec}{PSpec}
\DeclareMathOperator{\Prim}{Prim}
\DeclareMathOperator{\Clo}{Clo}
\DeclareMathOperator{\nl}{nl}
\DeclareMathOperator{\Fix}{Fix}

\newcommand{\id}{\mathrm{id}}
\newcommand{\Nat}{\mathbb{N}}
\newcommand{\A}{\mathcal{A}}
\newcommand{\B}{\mathcal{B}}

\newcommand{\midd}{\; \middle| \;}

\usepackage{hyperref}

\begin{document}


\title[Probabilistic spectrum, primality and approximation]{Probabilistic equational spectrum, primality and approximation in finite algebras}

\author[C. Card\'o]{Carles Card\'o}
\address{Departament de Medicina, Àrea d’Estadística, Salut Pública i Epidemiologia \\Universitat Internacional de Catalunya, c/ Josep Trueta s/n,  \\Sant Cugat del Vallès,  08195, Spain. }
\email{ccardo@uic.es}





\begin{abstract}
We define the probability of an equation in a finite algebra as the proportion of tuples in its domain that satisfy it.
We call the probabilistic spectrum of an algebra the set of probability values obtained when the equation varies. We study fundamental properties of this spectrum, such as density and limit points, and show that its structure is related to several notions of primality of an algebra. We introduce a quantitative measure of primality $\Prim(\A)\in[0,1]$ that characterizes the functional approximation capacity. We show that the degree of primality is related to the size of the spectrum. We also prove that all non-primal two-element algebras satisfy the universal bound $\Prim(\A)\le 1/2$.
\end{abstract}
\maketitle


\section{Introduction}
 
In the 1960s, in a series of papers \cite{ET1, ET2, ET3, ET4, ET6, ET7}, P. Erd\H{o}s and P. Turán introduced probabilistic and asymptotic counting methods into group theory. Only a few years later, a single result consolidated probabilistic group theory. W.H. Gustafson \cite{Gustafson1973} proved that the probability that two randomly chosen elements of a finite group $G$ commute is equal to $k(G)/|G|$, where $k(G)$ denotes the number of conjugacy classes of $G$; see also \cite{MacHale1974}. A well-known consequence is that if the group is non-abelian, then this probability cannot exceed $5/8$.

There are several directions for extending this idea. One is to consider infinite groups. To equip an infinite group with a probability distribution, it is necessary to introduce a measure, typically the Haar measure, a question that Gustafson himself already considered from the outset.

Another line of generalization consists of considering the set of commuting probabilities obtained when all finite groups are taken into account and studying its properties, such as density and limit points. In this direction, K.S. Joseph \cite{Joseph1969,Joseph1977} formulated several conjectures that have since been confirmed; see \cite{Eberhard2015}. It is also possible to replace the group structure with a closely related structure, such as that of a semigroup, while maintaining the commutation property; see \cite{Ponomarenko2012}. More generally, one may consider other properties, such as the probability that a pair of elements generates the whole group; for a survey, see \cite{Dixon2002}.

It is also possible to place the problem in the framework of universal algebra. When considering more general algebras, and therefore richer signatures, other types of equations must be taken into account. 
This line of research, to the best of our knowledge still little explored, is the one we develop in this article. In this direction, the interesting situation arises when one studies the set of probability values obtained by fixing an algebra and varying the equation. We call this set the \emph{equational probabilistic spectrum} of the algebra. This approach is orthogonal to that of K. S. Joseph and other authors, who fix the commutation equation $xy\approx yx$ and vary the group in order to study the resulting probability values.

We also study the limit points and the density. We will see that the spectrum is related to several relaxed notions of primality, such as idemprimality and automorphism-primality. In particular, we introduce a quantitative measure of the primality of an algebra, $\Prim(\A)\in[0,1]$, which describes its capacity to approximate arbitrary functions. The algebras with $\Prim(\A)=1$ coincide exactly with the primal algebras. We will show that the size of the spectrum depends on the degree of primality, and we will establish connections with coding theory. Our final result establishes that non-primal two-element algebras satisfy the universal bound $\Prim(\A)\le 1/2$.

An interesting parallel problem is that of studying algebras whose spectrum is as small as possible. This question is developed in a companion work \cite{Cardo2026min} and lies outside the scope of this article.

The structure of the article is as follows. In Section~\ref{ElementaryResults}, by way of motivation, we present some elementary examples of the computation of probabilities in concrete algebras. In Section~\ref{EspectreProbabilistic} we formally introduce the probabilistic spectrum and examine its first properties. The main results are found in Sections~\ref{LimitPoints},~\ref{GrausAprox}, and \ref{MidaPrim}. Finally, in the last Section~\ref{Barrera} we present the general result concerning the approximation of Boolean functions.

We fix some conventions of notation. We write algebras with calligraphic letters $\mathcal{A}$, except when referring to well-known algebras, such as $\mathbb{Z}_p$. Unless otherwise indicated, the corresponding italic letter $A$ will denote the underlying set of the algebra $\mathcal{A}$. All algebras considered will be finite and have a finite signature.

Given a fixed signature, by a \emph{term} we mean an element of the absolutely free algebra over an infinite set of variables. Given an algebra $\A$ with signature $\sigma$, a term $t^\A$ is a function, or operation, $t^\A:A^k\longrightarrow A$ obtained by interpreting in the algebra $\A$ the term $t$ with $k$ variables of the signature $\sigma$. When there is no ambiguity, we write $t$ instead of $t^\A$. For the definitions of lattices, groups, and other common concepts in algebra, we refer to \cite{Burris2012}. The remaining notions will be introduced throughout the text.

\section{Equational probability} \label{ElementaryResults}

By an \emph{equation} we mean an ordered pair of terms $(t,t')$ in a given signature, which we write as $t\approx t'$. The number of variables of an equation is the total number of distinct variables that appear in it. Given an equation $t\approx t'$ with $k$ variables over an algebra $\A$, we write the set of its solutions as
$$\{t \approx t'\}_{\A}= \{ \vec{x} \in A^k \mid t(\vec{x})=t'(\vec{x})\}.$$

\begin{definition}
We define the \emph{probability of the equation $t\approx t'$ over the algebra $\A$} as the fraction
$$\Pr (t\approx t' \mid \A)=\frac{|\{t \approx t'\}_\A|}{|A|^k}.$$
\end{definition}

Recall that all algebras considered are finite, so the above ratio is always well defined. We may also consider polynomial equations by incorporating constants into the signature, that is, by adding nullary operations corresponding to the desired elements.

Computing probabilities is closely related to the problem of counting solutions of the equation. In general this is difficult, even for well studied structures. If we consider, for instance, the case of groups (see \cite{Roman2012} for a survey), only partial results are known. From the computational point of view, we also lack efficient general methods for computing equational probabilities. We simply observe that the classical Boolean satisfiability problem $\mathsf{SAT}$ is equivalent to deciding whether $\Pr(t\approx 1 \mid \mathbf{2})>0$, given a term $t$ of the Boolean algebra $\mathbf{2}=(\{0,1\}, \wedge, \vee, \neg, 0,1)$.

Nevertheless, we can perform some simple computations if we have some knowledge of the structure of the algebra. The following Lemma~\ref{LemmaElementary} relates probabilities with algebra homomorphisms and direct products.

\begin{lemma} \label{LemmaElementary}
Let $\A,\B$ be algebras and let $t\approx t'$ be an equation with $k$ variables. The following properties hold.
\begin{enumerate}
\item[(i)] If $f: A\longrightarrow B$ is a monomorphism of algebras,
$$\Pr( t\approx t' \mid \A) \leq \left( \frac{|B|}{|A|}\right)^k \Pr( t\approx t' \mid \B). $$

\item[(ii)] If $f: A\longrightarrow B$ is an epimorphism of algebras,
$$1-\Pr( t \approx t' \mid \A) \geq \left( \frac{|B|}{|A|}\right)^k (1-\Pr( t\approx t' \mid \B)). $$

\item[(iii)] $\Pr( t\approx t' \mid \A\times \B) = \Pr( t\approx t' \mid \A) \cdot\Pr( t\approx t' \mid \B)$.
\end{enumerate}
\end{lemma}

\begin{proof}
\begin{enumerate}
\item[(i)] Since $f$ is a homomorphism, if $t(x_1, \ldots, x_k)=t'(x_1, \ldots, x_k)$ in $\A$, then $t(f(x_1), \ldots, f(x_k))=t'(f(x_1), \ldots, f(x_k))$ in $f(A)$. Since $f$ is injective,
$$|\{t\approx t'\}_\A| =|\{t\approx t'\}_{f(\A)}| \leq |\{t\approx t'\}_\B|.$$

Dividing by $|A|^k|B|^k$ we obtain
$$ \frac{|\{t\approx t'\}_\A|}{|A|^k|B|^k} \leq \frac{|\{t\approx t'\}_\B|}{|A|^k|B|^k}.$$
That is,
$$\frac{\Pr( t\approx t' \mid \A)}{|B|^k} \leq \frac{\Pr( t\approx t' \mid \B)}{|A|^k}.$$

\item[(ii)] First note that if $(f(x_1), \ldots, f(x_k))$ does not satisfy the equation in $\B$, then $(x_1, \ldots, x_k)$ does not satisfy the equation in $\A$. Since $f$ is surjective,
$$|\{ \vec{y} \in B^k \mid t(\vec{y}) \neq t'(\vec{y})\}|\leq |\{ \vec{y} \in A^k \mid t(\vec{y}) \neq t'(\vec{y})\}|.$$
The first set is in fact $B^k \setminus \{t\approx t'\}_\B$ and the second, $A^k \setminus \{t\approx t'\}_\A$. Therefore,
$$|B|^k - |\{t\approx t'\}_\B| \leq |A|^k - |\{t\approx t'\}_\A|.$$
Dividing by $|A|^k$ and $|B|^k$ and rearranging,
$$\frac{|B|^k - |\{t\approx t'\}_\B| }{|A|^k |B|^k} \leq \frac{|A|^k - |\{t\approx t'\}_\A|}{|A|^k |B|^k}.$$
Simplifying,
$$\frac{1}{|A|^k} (1-\Pr(t\approx t' \mid \B) ) \leq \frac{1}{|B|^k} (1-\Pr(t\approx t' \mid \A) ).$$

\item[(iii)] We simply observe that the sets $\{t\approx t'\}_{\A \times \B}$ and $\{t\approx t'\}_\A \times \{ t\approx t'\}_{\B}$ are bijective since $t^{\A\times \B}=(t^\A,t^\B)$.
\end{enumerate}
\end{proof}

Some of these bounds can be refined if we know more about the algebras or about the homomorphisms. For instance, it is not difficult to prove that if an epimorphism $f: \A \longrightarrow \B$ satisfies that the number $\kappa(f)=|f^{-1}(x)|$ is constant and does not depend on the element $x \in B$, then
$$\Pr(t\approx t' \mid \A) \leq \Pr(t\approx t' \mid \B)\leq \kappa(f)\Pr(t\approx t' \mid \A).$$
This is always the case for groups, where $\kappa (f)=|\ker (f)|$.

Let us consider some elementary applications of Lemma~\ref{LemmaElementary} in the case of lattices. We will use as an example the equation $x\wedge y\approx 0$, although the reader can repeat the computations with any other equation.

\begin{example}
We show that the probability that two sets are disjoint is always less than or equal to $3/4$. Let $U$ be a set with $m\geq 1$ elements. The algebra of subsets of $U$ is isomorphic to the direct product of $m$ copies of the two-element Boolean algebra $\mathbf{2}=(\{0,1\}, 0,1,\neg,\wedge, \vee)$. Since in $\mathbf{2}$ there are four possible pairs and three satisfy $x\wedge y=0$, by Lemma~\ref{LemmaElementary}(iii) we have
$$\Pr( x\wedge y\approx 0 \mid \mathbf{2}^m) = (\Pr( x\wedge y\approx 0 \mid \mathbf{2}))^m=\left( \frac{3}{4} \right)^m\leq \frac{3}{4}.$$
\end{example}

\begin{example}
Every non-modular lattice contains a copy of the pentagon lattice $\mathcal{N}_5$; see, for example, \cite{Burris2012}. Therefore, by Lemma~\ref{LemmaElementary}(i), for any equation over a non-modular lattice $\mathcal{L}$ with $n$ elements,
$$\Pr(t\approx t'\mid \mathcal{N}_5)\leq \left( \frac{n}{5}\right)^k \Pr(t\approx t' \mid \mathcal{L}).$$
In particular, for $n\geq 5$ and for the equation $x\wedge y \approx 0$, we can compute explicitly $\Pr(x\wedge y \approx 0 \mid \mathcal{N}_5)=14/25$ and obtain
$$\Pr(x\wedge y \approx 0 \mid \mathcal{L}) \geq \frac{14}{n^2}.$$
\end{example}

\begin{example}
Every distributive lattice $\mathcal{L}$ generated by at most $s$ generators is a quotient of the free lattice $\mathcal{FL}(s)$. Therefore, by Lemma~\ref{LemmaElementary}(ii), if $\mathcal{L}$ has $n$ elements, where necessarily $n\leq |\mathcal{FL}(s)|$,
$$\Pr(t\approx t'\mid \mathcal{L})\geq 1-\left( \frac{|\mathcal{FL}(s)|}{n}\right)^k (1- \Pr(t\approx t' \mid \mathcal{FL}(s))).$$
In particular, for two generators we have $|\mathcal{FL}(2)|=6$, and therefore $n\leq 6$. A direct inspection of the Hasse diagram (for instance in \cite{Jakel2023}) shows that $\Pr(x\wedge y\approx 0 \mid \mathcal{FL}(2))=13/36$. Hence,
$$\Pr(x\wedge y\approx 0 \mid \mathcal{L}) \geq 1-\frac{36}{n^2}(1-\frac{13}{36})=1-\frac{23}{n^2}.$$
However, in order for the bound to be nontrivial, it is necessary that $n\geq 5$, otherwise $1-\frac{23}{n^2}<0$. Therefore, for $n=5$,
$$\Pr(x\wedge y\approx 0 \mid \mathcal{L}) \geq \frac{2}{25}=0.08.$$
\end{example}

\section{The equational probabilistic spectrum of an algebra} \label{EspectreProbabilistic}

From now on, we focus on the study of the probabilistic spectrum: we fix an algebra and consider the probability values obtained when its equations vary. More formally,

\begin{definition}
We define the \emph{equational probabilistic spectrum of the algebra $\A$}, or more concisely, the \emph{spectrum of $\A$}, as
$$\PSpec(\A)=\{ \Pr (t\approx t' \mid \A) \mid  \mbox{all equations } t\approx t' \mbox{ in its signature}\}.$$
\end{definition}

Let us see some immediate properties. This set always contains at least the values $1$ and $1/|\A|$, since we may always consider the equations $x\approx x$ and $x\approx y$.
Naturally, two isomorphic algebras yield the same spectrum, but it is also easy to see that two anti-isomorphic groupoids have the same spectrum. Indeed, given a term $t$ of a groupoid, we can define its anti-term $\bar{t}$ by reversing the order of the operands. If $\bar{\A}$ is a groupoid anti-isomorphic to $\A$, then $\Pr(t\approx t'\mid \A)=\Pr(\bar{t}\approx \bar{t}'\mid \bar{\A})$. Since there is a natural bijection between the equations of $\mathcal A$ and those of $\bar{\mathcal A}$, we have $\PSpec(\A)=\PSpec(\bar{\A})$.

Another immediate property is that, by Lemma~\ref{LemmaElementary}(iii),
$$\PSpec(\A \times \B) \subseteq\PSpec(\A)\cdot \PSpec(\B),$$
where the product is pointwise, that is $X\cdot Y=\{xy \mid x\in X, y \in Y\}$.
For a power, we have a more precise result:
$$\PSpec(\A^m)=\{\alpha^m \mid \alpha \in \PSpec(\A)\},$$
since $\alpha \in \PSpec(\A)$ if and only if $\alpha^m \in \PSpec(\A^m)$. 
This means, for example, that if there are algebras whose spectrum is not dense in the interval $[0,1]$, as we will see later, then their powers are not dense either.

The following question can be answered easily in the case of groupoids (or of any algebra $\A$ with a single non-nullary operation): when does $0$ belong to the spectrum of $\A$? Equivalently, when does an algebra have at least one equation without solutions?
We have that $0\in \PSpec(\mathcal A)$ if and only if $\mathcal A$ has no idempotent element, in the case of a single non-nullary operation.
Recall that an element $x \in A$ is said to be idempotent for the operation $f$ when $f(x,\ldots, x)=x$. In fact, if $\A$ has a single non-nullary operation with $i$ idempotent elements and $t\approx t'$ is an equation with $k$ variables,
$$\Pr(t\approx t' \mid \A) \geq \frac{i}{|\A|^k}.$$

An algebra $\A$ is said to be \emph{derived from $\B$} if both have the same universe and the operations of $\A$ are terms of $\B$. Since any equation of $\A$ can be expressed as an equation of $\B$, we have
$$\PSpec (\A) \subseteq \PSpec (\B).$$
The same inclusion also holds if $\A$ is a reduct of $\B$, that is, if $\A$ is obtained by removing some operations from the signature of $\B$.

Finally, one more property that can be verified immediately. Let $\Clo(\A)$ denote the clone generated by the operations of the algebra $\A$. If $\Clo (\A) \subseteq \Clo (\B)$, then $\PSpec (\A) \subseteq \PSpec (\B)$. Indeed, if $t$ is a term of $\A$, then there exists a term $w$ of $\B$ with the same interpretation, $w^{\B}=t^{\A}$. Therefore, for each probability value $\alpha=\Pr(t\approx t' \mid \A)$ there exist terms $w,w'$ such that
$\Pr(w\approx w' \mid \B)=\alpha$.

In general, computing the spectrum of an algebra is not easy. Let us, however, show some simple examples.

\begin{example}
Let $\mathcal{P}_n=(\{1, \ldots, n\}, *)$ with the projection operation $x*y=x$. Any term $t$ reduces to the first variable that appears on the left in its representation, $t(x_1, \ldots, x_k)=x_1$. Therefore,  there are essentially only two distinct equations in $\mathcal{P}_n$, namely $x\approx y$ and $x\approx x$, which yield only two probability values,
$$\PSpec(\mathcal{P}_n)=\left\{ \frac{1}{n},1 \right\}.$$
\end{example}

\begin{example} \label{PSpecBoole}
Let $\mathbf{2}=(\{0,1\}, 0,1,\neg, \vee, \wedge)$ be the two-element Boolean algebra. By the functional completeness of this algebra, we know that any function $f: \mathbf{2}^k \longrightarrow \mathbf{2}$ can be written as a combination of its basic operations. Consequently, any subset of $\mathbf{2}^k$ can be written as the set of solutions of an equation in $k$ variables of the form $f(x_1, \ldots, x_k)\approx 0$. Thus, for any integer $s$ with $0\leq s \leq 2^k$ there exists a Boolean function $f: \mathbf{2}^k \longrightarrow \mathbf{2}$ such that $s=|f^{-1}(0)|$ and therefore
$$ \PSpec(\mathbf{2}) =
\left\{\, \frac{s}{2^k} \midd 0 \le s \le 2^k,\; k \ge 0 \,\right\}.$$
That is, the spectrum of $\mathbf{2}$ is the set of dyadic numbers in the interval $[0,1]$, which is dense.
\end{example}

\begin{example} \label{ExZp}
Let $\mathbb{Z}_p=(\{0,1,\ldots, p-1\},+, -(\cdot), 0)$ be the cyclic group of prime order $p$, where $-(\cdot)$ denotes the unary operation $-(x)=-x$. Let us show that
$$\PSpec(\mathbb{Z}_p)=\left\{ \frac{1}{p},1 \right\}.$$
Every equation $t\approx t'$ holds if and only if $t-t'\approx 0$, and therefore every equation can be written in the form $a_1x_1+\cdots +a_kx_k= 0$. Suppose that at least one of the coefficients is nonzero; otherwise, we would have $t=t'$ and hence $\Pr(t\approx t' \mid \mathbb{Z}_p)=1$. Without loss of generality, suppose that this coefficient is $a_k$. Since $p$ is prime, this equation is equivalent to
$$-a^{-1}(a_1x_1+\cdots +a_{k-1}x_{k-1})=x_k.$$
Note that although the multiplicative inverse is not part of the signature, this poses no problem, as we only use the ring structure to identify a suitable element. This equation has $p^{k-1}$ solutions, since $x_k$ is determined by the values of $x_1, \ldots, x_{k-1}$, and therefore,
$$\Pr(a_1x_1+\cdots +a_kx_k\approx 0 \mid \mathbb{Z}_p)=\frac{p^{k-1}}{p^k}=\frac{1}{p}.$$
Thus $\mathbb{Z}_p$ has the smallest possible spectrum, but it is essential that $p$ be prime.
\end{example}

The action of the automorphism group imposes a first restriction on the probabilistic spectrum of an algebra.
The solutions of an equation with $k$ variables in an algebra $\A$ form a subset of $A^k$. Consider the action
$\cdot:\Aut (\A) \times A^k \longrightarrow A^k$, defined componentwise,
$$\varphi \cdot (x_1, \ldots, x_k)=(\varphi(x_1), \ldots, \varphi(x_k)).$$
We introduce the following notation. Given $x_1, \ldots, x_n \in \Nat$,
$$\sum^\circ (x_1, \ldots, x_n)=\left\{ \sum_{i\in I} x_i \midd I \subseteq \{1,\ldots, n\} \right\}.$$
And given subsets $X_1, \ldots, X_k \subseteq \Nat$, we define
$$\ell(\{X_1, \ldots, X_k\})=(|X_1|,\ldots, |X_k|). $$
The following result gives a restriction on the spectrum based on the action of automorphisms.

\begin{theorem} \label{TeoOrbites}
For any algebra $\A$,
$$\PSpec( \A)\subseteq \bigcup_{k\geq1} \frac{1}{|A|^k}  \sum^\circ \ell ( A^k/\Aut(\A)).$$
\end{theorem}

\begin{proof}
Let $t\approx t'$ be an equation with $k$ variables. The set $A^k$ decomposes into orbits under the action
$A^k=C_1\cup \cdots \cup C_s$ where $A^k/\Aut(\A)=\{C_1, \ldots, C_s\}$. Let $\varphi$ be any automorphism. We have that $\vec{x}$ satisfies the equation if and only if $\varphi \cdot \vec{x}$ satisfies it, since automorphisms preserve terms. In other words, if one element of an orbit satisfies the equation, then the rest of the elements of the orbit also satisfy it. This implies that the set of solutions can be decomposed as a union of orbits:
$$\{t\approx t'\}_\A = C_{i_1} \cup \cdots \cup C_{i_r},$$
for some subset $I=\{i_1, \ldots, i_r\} \subseteq\{1,\ldots, s\}$. Therefore,
\[
\Pr (t\approx t' \mid \A)= \frac{|\{t\approx t'\}_\A|}{|A|^k}
= \frac{1}{|A|^k} \sum_{i\in I} | C_{i}| \,\, \in \frac{1}{|A|^k}  \sum^\circ \ell ( A^k/\Aut(\A)).\qedhere
\]
\end{proof}

The following example, which is very simple, illustrates an application of Theorem~\ref{TeoOrbites}.

\begin{example} \label{ExempleM3}
Consider the lattice $\mathcal{M}_n=(\{0,1,a_1, \ldots, a_n\}, \vee, \wedge)$; see Figure~\ref{FigOrbites}.
Consider an equation $t\approx t'$ with two variables.
On the one hand, we have that $\Aut(\mathcal{M}_n)\cong \mathfrak{S}_n$, where $\mathfrak{S}_n$ is the symmetric group of order $n!$, and that each automorphism fixes $0$ and $1$, and permutes the intermediate elements. The scheme in Figure~\ref{FigOrbites} shows the orbits. Hence,
$$\ell( M_n^2 / \Aut(\mathcal{M}_n))=(1,1,1,1,n,n,n,n,n, n^2-n).$$
Dividing by $(n+2)^2$ and applying $\sum^\circ$ we obtain that the probability $\Pr(t\approx t' \mid \mathcal{M}_n)$ must be given by a certain combination of the form
$$\alpha \frac{1}{(n+2)^2}+ \beta \frac{n}{(n+2)^2} + \gamma  \frac{n^2-n}{(n+2)^2},$$
where $ 0\leq \alpha\leq 4$, $0 \leq \beta \leq 5$, $0 \leq \gamma \leq 1 $.

For $n\leq5$, the previous expression fills all numerators $d/(n+2)^2$, $0\leq d \leq (n+2)^2$. However, for $n>5$ gaps begin to appear among the numerators. For $n=6$, one can verify manually that the probability of an equation with two variables cannot be a fraction $d/64$ with $d\equiv 5 \pmod 6$. Nevertheless, this is only a combinatorial restriction, and in fact, some other fractions are not realized as probabilities either.
\end{example}

\begin{figure}
\begin{tikzpicture}[scale=0.70] 
\tikzstyle{vertex}=[circle, draw, inner sep=0pt, minimum size=5pt]
  \node[vertex] (A) at (2,0) {};
  \node[vertex] (B) at (0,2)   {};
  \node[vertex] (C) at (1,2)   {};
  \node[vertex] (D) at (3,2)   {};
   \node[vertex] (E) at (4,2)   {};
   \node[vertex] (F) at (2,4)   {};
  \draw[thin] (A) -- (B) -- (F);
  \draw[thin] (A) -- (C) -- (F);
  \draw[thin] (A) -- (D) -- (F);
  \draw[thin] (A) -- (E) -- (F); 
  \node[align=left] at (1.5,0) {$0$};
  \node[align=left] at (-0.5,2) {$a_1$};
  \node[align=left] at (0.6,2) {$a_2$};
  \node[align=left] at (2,2) {$\cdots$};
  \node[align=left] at (4.5,2) {$a_{n}$};
  \node[align=left] at (1.5,4) {$1$};
 \end{tikzpicture}
 
{\footnotesize
\renewcommand{\arraystretch}{2.4}
\begin{tabular} {   | c | c c c | c |}
 \hline 
  $(0,0)$ & $(0,a_1)$ & $\cdots$ & $(0,a_n)$ & $(0,1)$  \\
\hline
 $(a_1,0)$ & $(a_1,a_1)$ & \cellcolor{lightgray} $\cdots$ &  \cellcolor{lightgray} $(a_1,a_n)$ & $(a_1,1)$  \\
$\vdots$ &  \cellcolor{lightgray}  $\vdots$ &  $\ddots$ &  \cellcolor{lightgray} $\vdots$ &  $\vdots$\\
 $(a_n,0)$ &  \cellcolor{lightgray} $(a_n,a_1)$ &  \cellcolor{lightgray} $\cdots$ & $(a_n,a_n)$ & $(a_n,1)$  \\
\hline
 $(1,0)$ & $(1,a_1)$ & $\cdots$ & $(1,a_n)$ & $(1,1)$  \\
 \hline 
\end{tabular}}
\caption{The lattice $\mathcal{M}_n$ at the top of the figure and below a scheme of the orbits of $M_n^2/\Aut(\mathcal{M}_n)$ from Example~\ref{ExempleM3}. The orbits are separated by lines, with the exception of the center of the table, where the gray cells form a single orbit of length $n^2-n$, whereas the elements on the diagonal form an orbit of length $n$.}\label{FigOrbites}
\end{figure}
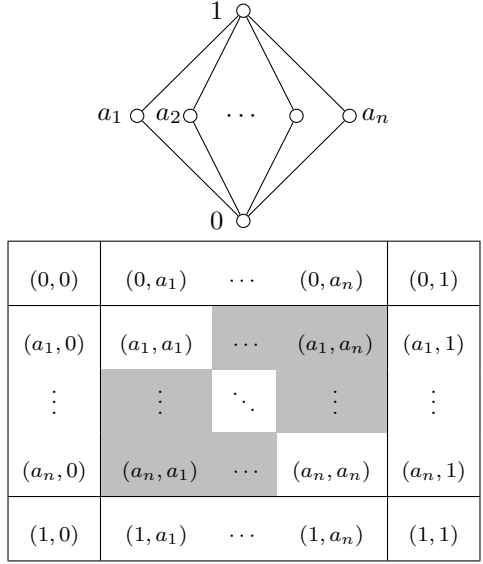

The natural question is when the inclusion of Theorem~\ref{TeoOrbites} is an equality. Given an algebra $\A$, if $t \in \Clo(\A)$, then $t(\varphi \cdot \vec{x})=\varphi (t(\vec{x}))$ for all $\varphi \in \Aut(\A)$. 
An algebra is said to be \emph{automorphism-primal} when the converse also holds; see \cite{KaarliPixley2000}. We denote by $\Fix(\A)$ the subalgebra of fixed points of $\A$, that is,
$$\Fix(\A)=\{ x\in A \mid \varphi(x)=x, \,\,\forall \varphi \in \Aut(\A)\}.$$

\begin{theorem}\label{TeoAutoPrimal}
If $\A$ is an automorphism-primal algebra such that $\Fix(\A)$ has at least two elements, then
$$\PSpec( \A)= \bigcup_{k\geq1} \frac{1}{|A|^k}  \sum^\circ \ell ( A^k/\Aut(\A)).$$
\end{theorem}

\begin{proof}
Let $a,b\in \Fix(\A)$ with $a\not=b$ and $|\A|=n$. Fix an arity $k$, and let $A^k/\Aut(\A)=\{C_1, \ldots, C_s\}$. Given $I \subseteq \{1,\ldots, s\}$, define the function
$$f_I(\vec{x})= \begin{cases} a &\mbox{ if } \vec{x} \in \bigcup_{i\in I} C_i,\\ b &\mbox{ otherwise.}\end{cases}$$

Now note that if $\varphi$ is any automorphism,
$$f_I(\varphi(\vec{x}))= \begin{cases} a &\mbox{ if } \varphi(\vec{x}) \in \bigcup_{i\in I} C_i,\\ b &\mbox{ otherwise.}\end{cases}$$

Note that $\varphi(C_i)=C_i$ and therefore $\varphi(\vec{x}) \in \bigcup_{i\in I} C_i$ if and only if $\vec{x} \in \bigcup_{i\in I} C_i$. Since $a$ and $b$ are fixed points,
$$f_I(\varphi(\vec{x}))=
\begin{cases} a &\mbox{ if } \vec{x} \in \bigcup_{i\in I} C_i,\\ b &\mbox{ otherwise.}\end{cases}
=
\begin{cases} \varphi(a) &\mbox{ if } \varphi(\vec{x}) \in \bigcup_{i\in I} C_i,\\ \varphi(b) &\mbox{ otherwise.}\end{cases}
=  \varphi ( f_I(\vec{x})).$$

Thus, since $\A$ is automorphism-primal, $f_I$ is a term of the algebra for every $I\subseteq \{1,\ldots, s\}$. Note also that $f_{\{1,\ldots, s\}}$ is the constant term $f_{\{1,\ldots, s\}}(\vec{x})=a$ for all $\vec{x}\in A^k$. Finally, it only remains to show that every probability value of 
$\frac{1}{n^k}  \sum^\circ \ell ( A^k/\Aut(\A))$ is realizable by an equation:
\begin{align*}
\Pr(f_I(\vec{x})\approx f_{\{1,\ldots, s\}}(\vec{x}) \mid \A)=
\frac{ |\bigcup_{i\in I} C_i |}{n^k}=
\sum_{i\in I}\frac{ |C_i|}{n^k}.
\end{align*}
\end{proof}

\section{Limit points and density} \label{LimitPoints}

Let $X\subseteq [0,1]$. A point $\alpha \in [0,1]$ is called a \emph{limit point} of $X$ if for every neighborhood $U_\alpha$ of $\alpha$, we have that
$X\cap (U_\alpha\setminus \{\alpha\})\neq\emptyset$.
We say that $X$ is dense in $[0,1]$ if every $\alpha\in[0,1]$ belongs to $X$ or else is a limit point of $X$. We begin with two examples showing the existence of limit points.

\begin{example}
Let $\mathcal{C}_2=(\{0,1\}, \cdot)$ be the semilattice, where
$0\cdot 0 = 0\cdot 1= 1\cdot 0=0$ and $1\cdot 1=1$.
By associativity, commutativity, and idempotence, any term is equivalent to a product of variables without repetition. Therefore, every equation in $\mathcal{C}_2$ is equivalent to one of the form $xy=xz$, where
$x=x_1\cdots x_r$, $y=y_1\cdots y_p$, $z=z_1\cdots z_q$, for some integers $p,q,r\geq 0$.
If $x,y,z$ are pairwise distinct variables, the solutions of $xy\approx xz$ are
$$\{xy\approx xz\}_{\mathcal{C}_2}=\{0,1\}^3 \setminus \{ (1,1,0), (1,0,1)\},$$
that is, the equation fails when $x=1$ and $y\neq z$.

Since in general $\Pr(x_1\cdots x_r \approx 1)=1/2^r$, and similarly for $y$ and $z$, we have that
\begin{align*}
\Pr(xy\approx xz\mid \mathcal{C}_2)
&=1-\Pr\Big( (x,y,z)\approx(1,1,0) \mbox{ or } (x,y,z)\approx(1,0,1)\mid \mathcal{C}_2\Big)\\
&=1-\frac{1}{2^r}\frac{1}{2^p}\left(1-\frac{1}{2^q}\right)
-\frac{1}{2^r}\frac{1}{2^q}\left(1-\frac{1}{2^p}\right)\\
&=1-\frac{2^p+2^q-2}{2^{p+q+r}}=\varphi(p,q,r).
\end{align*}

Therefore,
$$\PSpec(\mathcal{C}_2)=\left\{\varphi(p,q,r)  \mid p,q,r \geq 0\right\}.$$

It is easy to verify that if $r\geq0$ and $p,q\geq 2$, then, when fixing any pair of the three variables of $\varphi(p,q,r)$, the resulting function is strictly decreasing. This implies that $\PSpec(\mathcal{C}_2)$ contains a unique limit point, namely $0$ as $p,q,r \to \infty$. Although we have an infinite spectrum, it is not dense. For completeness, we summarize in Table~\ref{TaulaOrdre2} the known spectra of groupoids with two elements.
\end{example}

\begin{table}
\footnotesize
\centering

\begin{tabular}{c l}

\begin{tabular}{c|cc}
 &0&1\\\hline
0&0&0\\
1&0&0
\end{tabular}\quad
\begin{tabular}{c|cc}
 &0&1\\\hline
0&1&1\\
1&1&1
\end{tabular}
&
$\left\{\tfrac12,1\right\}$
\\
\midrule

\begin{tabular}{c|cc}
 &0&1\\\hline
0&0&0\\
1&1&1
\end{tabular}\quad
\begin{tabular}{c|cc}
 &0&1\\\hline
0&0&1\\
1&0&1
\end{tabular}
&
$\left\{\tfrac12,1\right\}$
\\
\midrule

\begin{tabular}{c|cc}
 &0&1\\\hline
0&0&1\\
1&1&0
\end{tabular}\quad
\begin{tabular}{c|cc}
 &0&1\\\hline
0&1&0\\
1&0&1
\end{tabular}
&
$\left\{\tfrac12,1\right\}$
\\
\midrule

\begin{tabular}{c|cc}
 &0&1\\\hline
0&0&0\\
1&0&1
\end{tabular}\quad
\begin{tabular}{c|cc}
 &0&1\\\hline
0&0&1\\
1&1&1
\end{tabular}
&
$\displaystyle \left\{1-\frac{2^p+2^q-2}{2^{p+q+r}} \midd p,q,r\ge0\right\}$
\\
\midrule

\begin{tabular}{c|cc}
 &0&1\\\hline
0&1&1\\
1&0&0
\end{tabular}\quad
\begin{tabular}{c|cc}
 &0&1\\\hline
0&1&0\\
1&1&0
\end{tabular}
&
$\left\{0,\tfrac12,1 \right\}$
\\
\midrule

\begin{tabular}{c|cc}
 &0&1\\\hline
0&1&1\\
1&0&1
\end{tabular}\quad
\begin{tabular}{c|cc}
 &0&1\\\hline
0&0&1\\
1&0&0
\end{tabular}\quad
\begin{tabular}{c|cc}
 &0&1\\\hline
0&1&0\\
1&1&1
\end{tabular}\quad
\begin{tabular}{c|cc}
 &0&1\\\hline
0&0&0\\
1&1&0
\end{tabular}
&
\textit{spectrum unknown}
\\
\midrule

\begin{tabular}{c|cc}
 &0&1\\\hline
0&1&0\\
1&0&0
\end{tabular}\quad
\begin{tabular}{c|cc}
 &0&1\\\hline
0&1&1\\
1&1&0
\end{tabular}
&$\displaystyle \left\{\frac{d}{2^k}\midd 0\le d\le 2^k,\;k\ge0\right\}$\\
\end{tabular}

\caption{Spectrum of all groupoids of order two, grouped by isomorphism and anti-isomorphism.}
\label{TaulaOrdre2}
\end{table}

\begin{example} \label{ExS3} We illustrate that computing the spectrum is, in general, nontrivial. Consider the case of the smallest non-abelian group. The total spectrum of $\mathfrak{S}_3$ is unknown. However, we can fix a family of non-trivial equations and compute their corresponding probabilities. The elements of the group can be written as
$$\mathfrak{S}_3=\{ 1, r,r^2, s, rs, r^2s\},$$
where $r$ and $s$ are the permutations $(3\,1\,2)$ and $(2\,1)$, respectively. First note that the squares in $\mathfrak{S}_3$ are the elements of the cyclic subgroup $\langle r \rangle$:
\begin{center}\renewcommand{\arraystretch}{1.2}
\begin{tabular}{ c | c c c c c c }
 $x$ &   $ 1$ & $ r$ & $r^2$ & $s$ & $rs$ & $r^2s$ \\ 
 \hline
 $x^2$ &   $ 1$ & $ r^2$ & $r$ & $1$ & $1$ & $1$ \\ 
\end{tabular}
\end{center}
We can rewrite the equation $x_1^2 \cdots x_k^2 \approx 1$ as a system of equations:
$$y_1\cdots y_k=1, \,\, y_1=x_1^2, \,\, \ldots,\,\, y_k=x_k^2.$$
The solutions of the first equation are of the form
$$\left\{\left(y_1, \ldots, y_{k-1}, (y_1 \cdots y_{k-1})^{-1}\right)\mid y_1, \ldots, y_{k-1}\in \{1,r,r^2\}\right\}.$$ 
If we only consider the first equation, we can express the last component in terms of the first $k-1$ ones. Each of the components satisfies that $y_j=x_j^2$. If $y_j=r$, then $x_j=r^2$. If $y_j=r^2$, then $x_j=r$. In contrast, if $y_j=1$, we have the possibilities $x_j \in \{1,s,rs,r^2s\}$. Therefore, in the general counting, we must multiply by four each possible $y_j$, that is, we must include a factor $4^{s_1}$, where $s_1$ is the number of times the identity permutation $1$ appears in the $k-1$ components. Moreover, we must take into account that the last component can also be equal to $y_k=1$, and when this is the case, it is necessary to multiply by four the number of solutions. However, this only occurs if the product of the first $k-1$ components is $1$, $y_1\cdots y_{k-1}=1$, and this happens if and only if the number of occurrences of the permutation $r$ in the first $k-1$ components, which we denote by $s_2$, and the number of occurrences of $r^2$, which we denote by $s_3$, satisfy that $s_2+2s_3\equiv 0\pmod 3$. This is equivalent to saying that $s_2\equiv s_3 \pmod 3$. Thus, using multinomial coefficients,
$$\Pr(x_1^2 \cdots x_k^2 \approx 1 \mid \mathfrak{S}_3)=\frac{1}{6^k}\sum_{s_1+s_2+s_3=k-1} { k-1 \choose s_1, s_2, s_3 } 4^{s_1 +\sigma(s_2,s_3)},$$
where
$$\sigma(s_2,s_3)=\begin{cases} 1 &\mbox{ if } s_2\equiv s_3 \pmod 3, \\ 0 &\mbox{otherwise.} 
 \end{cases}$$
We can eliminate the term $\sigma(s_2,s_3)$ using complex cubic roots of unity $\omega=e^{\frac{2\pi i}{3}}$. It holds that
 $$4^{\sigma(s_2,s_3)}=2+\omega^{s_2-s_3}+ \omega^{2(s_2-s_3)},$$
since
 $$1+\omega^k+\omega^{2k}=\begin{cases} 3 &\mbox{ if } k\equiv0\pmod3,\\ 0 & \mbox{ otherwise.}\end{cases}$$ 
Substituting this into the sum, using the multinomial theorem and making the corresponding simplifications, we obtain that
\begin{align*}
&\sum C 4^{s_1}\left(2+\omega^{s_2-s_3}+ \omega^{2(s_2-s_3) }\right)\\
=\,\,& 2\sum C 4^{s_1} +\sum C 4^{s_1}\omega^{s_2}(\omega^{-1})^{s_3}+\sum C4^{s_1} (\omega^2)^{s_2}(\omega^{-2})^{s_3}  \\
=\,\,& 2(4+1+1)^{k-1}+(4+\omega + \omega^{-1} )^{k-1} + (4+\omega^2+\omega^{-2})^{k-1} \\
 =\,\, & 2 \cdot 6^{k-1}+(4-1)^{k-1}+(4-1)^{k-1}=2(6^{k-1}+3^{k-1}),
\end{align*}
where $C={ k-1 \choose s_1, s_2, s_3 }$ and the sums run over $s_1+s_2+s_3=k-1$. Therefore, returning to the initial computation,
$$\Pr(x_1^2 \cdots x_k^2 \approx 1 \mid \mathfrak{S}_3) =\frac{1}{3}\left( 1+ \frac{1}{2^{k-1}}\right).$$
This tells us that the spectrum of $\mathfrak{S}_3$ is infinite and that it has a limit point at $1/3$. 
\end{example}

From Example~\ref{ExZp} we know that the spectrum of the group $\mathbb{Z}_p$ takes only two values, $1$ and $1/p$, when $p$ is prime. However, if we consider the ring $\mathbb{Z}_p$ with the two usual operations and the constants $0$ and $1$, then its spectrum is dense.
This is due to the fact that these structures, like the Boolean algebra $\mathbf{2}$, are primal; see Example~\ref{PSpecBoole}. Recall that an algebra is called \emph{primal} when any operation of positive arity can be expressed as a term of the algebra. The following theorem is a direct generalization of Example~\ref{PSpecBoole}, which we prove for completeness.

\begin{theorem} \label{TeoDensPrimal} The spectrum of any non-trivial primal algebra of order $n$ is the set of the $n$-adic numbers in the interval $[0,1]$.
\end{theorem}
\begin{proof}
Since it is primal, fix an element $b \in A$ and consider the constant unary term $f_b(x)=b$ for all $x\in A$. On the other hand, since $\A$ is non-trivial there exists an element $a\neq b$, and given any subset $B\subseteq A^k$ the indicator function defined as
$$f_B(\vec{x})=\begin{cases} b & \mbox{ if } \vec{x} \in B, \\ a & \mbox{ otherwise}; \end{cases}$$
is a term of the algebra. Therefore, $ \Pr(f_B\approx f_b \mid \A)=\frac{|B|}{n^k}$.
Since we can choose $B$ with any desired size, we obtain all $n$-adic numbers.
\end{proof}

Although primality guarantees the density of the spectrum, the converse statement is no longer true. Recall that
$\PSpec (\A^m)= \{ \alpha^m \mid \alpha \in \PSpec(\A)\}$. It is easy to see that the algebra $\A$ has dense spectrum if and only if $\A^m$ does as well.
Thus, although the Boolean algebra $\mathbf{2}^m$ has dense spectrum, $\mathbf{2}^m$ is not primal for $m>1$.

There is a generalization of the previous Theorem~\ref{TeoDensPrimal}, with more interesting consequences. A function $f$ is \emph{idempotent} if $f(x, x, \ldots, x)=x$. An algebra is called \emph{idemprimal} if every idempotent function is a term.

\begin{theorem} \label{TheoIdemPrimal}
The spectrum of every non-trivial idemprimal algebra is dense.
\end{theorem}

\begin{proof}
Let $\A$ be an algebra of order $n$. Denote the diagonal subset of $A^k$ by $\Delta_k$. Fix an arity, and let $B$ be a subset subject only to the condition $B \subseteq A^k\setminus \Delta_k$. Let $a,b \in A$ with $a\not=b$. Define the pair of functions
$$f_b(\vec{x})=\begin{cases} 
b &\mbox{ if } \vec{x}\not \in \Delta_k, \\ 
x_1 & \mbox{ if } \vec{x}\in \Delta_k;
\end{cases} 
\qquad \quad
f_B(\vec{x})=\begin{cases} 
b &\mbox{ if } \vec{x} \in B, \\ 
a & \mbox{ if } \vec{x} \in (A^k\setminus \Delta_k)\setminus B, \\ 
x_1 & \mbox{ if } \vec{x}\in \Delta_k;
\end{cases} $$
where $\vec{x}=(x_1, \ldots, x_k)$. We have $\{ f_B\approx f_b\}_\A=\Delta_k \cup B$. Therefore,
$$\frac{1}{n^{k-1}}\leq \Pr(f_B\approx f_b\mid \A)=\frac{n+|B|}{n^k}\leq 1,$$
where the inequalities follow from the fact that $0\leq |B| \leq n^k-n$. By varying the size of $B$ within these bounds, we obtain that the set of probabilities is dense. Note that the value $0$, although it does not belong to the spectrum, appears as a limit point as the arity tends to infinity.
\end{proof}

Theorem~\ref{TheoIdemPrimal} actually tells us a much more general fact. V. L. Murski\v{\i} \cite{Murski1975} proved that “almost” all algebras with at least one operation of arity two or greater are idemprimal; for an updated proof, see \cite{Freese2022} or \cite{Bergman2012}. Here, “almost” is taken in a different probabilistic sense, not an equational one. An algebraic property is said to hold \emph{almost always} if the proportion of algebras satisfying the property among all algebras of size $n$ tends to $1$ as $n$ tends to infinity; see \cite{Freese1990}. We therefore obtain the following immediate consequence.

\begin{corollary}
Almost all algebras with at least one operation of arity two or greater have a dense spectrum.
\end{corollary}

As for the automorphism-primal algebras mentioned in the previous section, it is not difficult to see that if they have two fixed points, then as a consequence of Theorem~\ref{TeoAutoPrimal} their spectrum is infinite. Moreover, both $0$ and $1$ are limit points.

\section{Degrees of primality and approximation} \label{GrausAprox}
We have seen some relaxed or relativized forms of primality, such as idempotent-primality or automorphism-primality; see \cite{KaarliPixley2000} for some other forms of primality. The probabilistic framework developed in this article suggests another variant, in this case, quantitative. We define the coincidence ratio between two functions of the same arity $f,g: A^k \longrightarrow A$ as
$$\mu(f,g)=\frac{|\{ \vec{x} \in A^k \mid f(\vec{x})=g(\vec{x})\}|}{|A|^k},$$
whenever $k\geq 1$. For convenience, we do not define the coincidence ratio for $k=0$.

Recall that $\Clo(\A)$ is the clone of functions generated by the operations of $\A$. We denote by $\Clo_k(\A)$ the functions in $\Clo(\A)$ of arity $k$. We denote by $\mathcal{F}(\A)$ the set of all finitary functions with universe $A$, and by $\mathcal{F}_k(\A)$ the subset restricted to functions of arity $k$.

\begin{definition}
For each $k\geq 1$, we define the \emph{arity-$k$ primality} of an algebra $\A$ as
$$\Prim_k(\A)=\min_{f \in \mathcal{F}_k(\A)}\,\, \max_{t \in \Clo_k(\A)} \mu(f,t).$$
We define the \emph{primality} of $\A$ as the number
$$\Prim(\A)=\inf_{k\geq 1} \Prim_k(\A).$$
\end{definition}

We observe that we only consider functions of positive arity, in coherence with the usual definition of primality. Note that the primality of an algebra $\A$ is always well defined, since $0\leq \Prim_k (\A) \leq 1$, and therefore, the infimum always exists. We have that $\Prim (\A)=1$ if and only if $\A$ is primal. On the other hand, it is easy to find algebras with zero primality. For example, for the semigroup given by the cyclic group, we have $\Prim (\mathbb{Z}_n)=0$. To prove this, it suffices to observe that the only unary term is the identity (since no nontrivial constants or translations are available). Then, $\mu(\mathrm{id}, \mathrm{id}+1)=0$, and since primality since primality is defined as the minimum over the best possible approximations, its primality is $0$.

Thus, the numerical interpretation of $\Prim(\A)$ is that algebras with primality close to $1$ are good function approximators, and conversely, those with value $0$ are not.

There is another interpretation, in this case geometric. Let us first note that, given two functions $f,g: A^k\longrightarrow A$, the function
$$D(f,g)=|\{\vec{x}\in A^k \mid f(\vec{x})\not=g(\vec{x})\}|$$
is a distance function. This can be seen by identifying each function $f:A^k\longrightarrow A$ with a string of length $|A|^k$ over the alphabet $A$, and $D$ is the Hamming distance over the alphabet $A$; see for example \cite{Pless1998}. Therefore, the normalized function
$$d(f,g)=\frac{1}{|A|^k}D(f,g)$$
is also a distance function that measures the error or discrepancy between the two functions and satisfies that
$$\mu(f,g)=1-d(f,g) \in [0,1].$$
Given a set of functions $\mathcal{C} \subseteq \mathcal{F}(\A)$, the distance from a function $f$ to the set $\mathcal{C}$ is defined as
$$d(f,\mathcal{C})=\min_{g \in \mathcal{C}} d(f,g).$$
Then we have that
$$\Prim(\A)=1-\max_{f\in \mathcal{F}(\A)} d(f,\Clo(\A)).$$
That is, in an algebra with primality $\varepsilon$, any function is at distance at most $1-\varepsilon$ from some term of the algebra.

\begin{example} In general, groups are not good function approximators. If $\mathcal{G}$ is a non-cyclic group, then
$$\Prim (\mathcal{G}) =0.$$
Let us see this briefly. The unary clone of a group consists of the functions $t(x)=x^m$. On the other hand, note that if the group is not cyclic, then for every element $x$ of the group, there always exists an element that is not a power of $x$, since otherwise we would have that $\langle x \rangle =G$.
We can construct a unary function $g$ that disagrees at all points with any power function. For each $x\in G$ choose an element $y_x \not \in \langle x \rangle$ and define $g(x)=y_x$. We then have that $\mu(g(x), x^m)=0$, for all $m$. Thus, $\Prim (\mathcal{G})\leq \Prim_1  (\mathcal{G})=0$.
It is worth noting that for cyclic groups, the situation becomes different.
\end{example}

\begin{example} \label{ClonAfi}
We denote by $\mathbb{Z}_n^+$ the cyclic group $\mathbb{Z}_n$ enriched with all nullary functions. The clone of this algebra is the affine clone, that is, the functions of the form
$$f(x_1,\ldots, x_k)=a_1x_1 +\cdots +a_kx_k+ b,$$
with $a_1, \ldots, a_k, b \in \mathbb{Z}_n$.
For the case of order two and even arity $k$, we know the exact primalities:
$$\Prim_k(\mathbb{Z}_2^+)=\frac{1}{2}+\frac{1}{2^{\frac{k}{2}+1}},$$
and for odd arity, we have the inequalities
$$\frac{1}{2}+\frac{1}{2^{\frac{k}{2}+1}}\leq \Prim_k(\mathbb{Z}_2^+)\leq \frac{1}{2}+\frac{1}{2^{\frac{k+1}{2}}}.$$
These expressions come from the concept of the \emph{nonlinearity} of Boolean functions. Nonlinearity is defined as the Hamming distance of $f$ to the affine clone, denoted by $\nl_k(f)$. From coding theory we know that if $\nl_k=\max_{f  \in \mathcal{F}_k} \nl_k(f)$,
$$\nl_k=2^{k-1}-2^{\frac{k}{2}-1},$$
in the case of even arity, and that
$$2^{k-1}-2^{\frac{k-1}{2}}\leq \nl_k \leq 2^{k-1}-2^{\frac{k}{2}-1},$$
in the odd case. That is,
$$\Prim_k(\mathbb{Z}_2^+)=1-\frac{\nl_k}{2^k}.$$
For the even case, the nonlinearity is achieved by the so-called bent functions, functions well studied in cryptography, whereas, for the odd-arity case only a few cases are known; see \cite{carlet2010boolean}. In fact, the number $\nl_k$ coincides with the covering radius of a Reed–Muller code of order 1; see \cite{Pless1998}.

Returning to our framework, the set of $k$-ary primalities forms a decreasing sequence and therefore,
$$\Prim(\mathbb{Z}_2^+)=\frac{1}{2}.$$
\end{example}

Although it is not trivial to calculate the primality of an algebra, we can provide some bounds. We begin with an upper bound for primality, which is easy to prove; see later, however, Section~\ref{Barrera}.

\begin{theorem} \label{CotaFeble}
If $\A$ is a non-primal algebra of order $n$, then
$$\Prim(\A)\leq 1-\frac{1}{n^2}.$$
\end{theorem}

\begin{proof}
According to a classical result of W. Sierpiński \cite{Sierpinski1945}, for every finite set $A$, every function can be expressed as a finite composition of binary operations. Thus, if $\A$ is not primal, there exists at least one function $f$ of arity at most $2$ that does not belong to $\Clo(\A)$. If $k=1$, then for every unary term of the algebra, $t\neq f$. Therefore, $f$ and $t$ disagree at least at one point, that is, $\mu(f,t)\leq (n-1)/n$. If $k=2$, the argument is the same, but now $\mu(f,t)\leq (n-1)/n^2$. Since $1-1/n \leq 1-1/n^2$ for all $n\geq 1$, we have that $\Prim(\A)\leq \Prim_2(\A)\leq 1-1/n^2$.
\end{proof}

\begin{example} \label{ExIdemPrimal}
If $\A$ is an idemprimal algebra, then for each $k>1$
$$\Prim_k (\A) \geq 1-\frac{1}{n^{k-1}}.$$
This is due to the fact that for each function $f$ of arity $k>1$ we can choose a term that agrees with $f$ at all points except at the $n$ points of the diagonal $\Delta_k$. Therefore we always have a minimal coincidence $\mu(f,t)\geq \frac{n^k -n}{n^k}=1-\frac{1}{n^{k-1}}$.

We note a seemingly paradoxical phenomenon. We have that
$$\lim_{k\to\infty} \Prim_k(\A)=1.$$
However, it may happen that the global primality is zero, $\Prim(\A)=0$. This is because an idemprimal algebra may fail to approximate unary functions. In other words, idemprimal algebras give good approximations only for large arities.
\end{example}

\begin{proposition}\label{Primaldos}
Let $\A$ be an algebra of order $n$ with at least one operation of arity greater than or equal to two, and let $\rho: A \longrightarrow A$ be a cyclic permutation. Let $\A^\rho$ be the algebra $\A$ enriched with the unary operation $\rho$. We have that for all $k\geq 1$,
$$\Prim_k (\A^\rho)\geq \frac{1}{n}.$$
\end{proposition}

\begin{proof}
Consider the Kronecker delta function $\delta: A^2 \longrightarrow \{0,1\}$, defined as $\delta(x,y)=1$ if $x=y$, and $\delta(x,y)=0$ if $x\not=y$. Since the algebra has an operation of arity $\geq2$, each set $\Clo_k(\A)$ is non-empty. Fixing an arity, take a $k$-ary term of the algebra and also $f$ any $k$-ary function. Now note that since $\rho$ is a cycle of order $n$, for a fixed $\vec{x} \in A^k$, there exists $i\in \{0,\ldots, n-1\}$ such that
$$ \delta \big( f(\vec{x}),(\rho^i \circ t)(\vec{x}) \big)=1 \quad \mbox{and} \quad  \delta \big( f(\vec{x}), (\rho^j\circ t)( \vec{x}) \big)=0, \,\, \forall j \in \{0,\ldots, n-1\} \setminus \{i\}.$$
Now consider the quantity $S$, which equals 1:
$$S=\frac{1}{n^k} \sum_{\vec{x}\in A^k}\sum_{i=0}^{n-1} \delta \big( f(\vec{x}), (\rho^i \circ  t)(\vec{x}) \big)= \frac{1}{n^k} \sum_{\vec{x}\in A^k}1=\frac{1}{n^k}n^k=1.$$
However, $S$ can be computed in another way. Interchanging the sums and the factor $1/n^k$, we obtain
$$S=\sum_{i=0}^{n-1} \frac{1}{n^k} \sum_{\vec{x}\in A^k} \delta \big( f(\vec{x}), (\rho^i \circ t)(\vec{x})\big)=\sum_{i=0}^{n-1} \mu(f,\rho^i \circ t).$$
That is,
$$\sum_{i=0}^{n-1} \mu(f, \rho^i \circ t) =1.$$
This means that we can find some $i=0, \ldots, n-1$ such that $ \mu(f,\rho^i \circ t)\geq 1/n$, otherwise the sum would be strictly less than 1. In other words, we can always approximate a function $f$ by some term $\rho^i \circ t$ in such a way that the coincidence ratio is greater than or equal to $1/n$.
\end{proof}

A. L. Foster \cite{Foster1953} proved that if we extend a group $\mathcal{G}=(G, \cdot, 1)$ with an absorbing element $0$, $G'=G \cup \{0\}$ (that is, $x\cdot 0=0\cdot x=0$), and with a cyclic permutation of $G'$, then the resulting algebra $\mathcal{G}'=(G', \cdot, 1, 0, \rho)$ is primal. This construction has some similarities with the algebras in Proposition~\ref{Primaldos}, but without the presence of the absorbing element or the group structure. Thus, the degree of primality can be very sensitive to small changes in the structure.

We also note that the clone generated by the algebra $\mathbb{Z}_n^+$ is the affine clone, which coincides with the clone of $\mathbb{Z}_n^\rho$. Applying Proposition~\ref{Primaldos},
$$\Prim_k( \mathbb{Z}_n^+)\geq \frac{1}{n}.$$

\section{Size of the spectrum and primality} \label{MidaPrim}

The Hamming metric on the space of functions allows us to establish a relationship between the size of the spectrum and primality.

\begin{lemma} \label{LemmaQuadrilater}
Let $\A$ be a non-trivial algebra with $\Prim_k(\A)=\varepsilon$. For every $\alpha \in [0,1]$ there exists an equation $t\approx t'$ with $k$ variables such that
$$\Pr(t\approx t' \mid \A ) \in [\alpha -2\bar{\varepsilon}, \alpha +2\bar{\varepsilon}],$$
where $\bar{\varepsilon}=1-\varepsilon$.
\end{lemma}

\begin{proof}
First, we prove the statement for $n$-adic numbers. That is, assume $\alpha \in [0,1]$ is $n$-adic, meaning of the form $r/n^k$, for some $0\leq r \leq n^k$ and $k\geq 0$. For such an $\alpha$, consider a function
$f_\alpha:A^k\longrightarrow A$ defined as
$$f_\alpha(\vec{x}) =
\begin{cases}
a & \text{if } \vec{x}\in R,\\
b & \text{otherwise},
\end{cases}$$
where $b\not= a$. Such a function is well defined since the algebra is non-trivial and contains at least two elements, $a$ and $b$. 
Denote by $f_a$ the constant $k$-ary function taking the value $a$, that is, $f_a(\vec{x})=a$ for all $\vec{x} \in A^k$. We have that $\mu(f_\alpha, f_a)=\alpha$, equivalently $d(f_\alpha, f_a)=1-\alpha$.

Since $\Prim_k(\A)=\varepsilon$, we can choose terms $t_\alpha$ and $t_a$ such that
$$\mu(f_\alpha,t_\alpha)\geq \varepsilon, \qquad \mu(f_a,t_a)\geq \varepsilon.$$
Writing $\bar{\varepsilon}=1-\varepsilon$, we obtain
$$d(f_\alpha,t_\alpha)\leq \bar{\varepsilon}, \qquad d(f_a,t_a)\leq \bar{\varepsilon}.$$
In the metric space of all functions endowed with distance $d$, consider the quadrilateral formed by the points $f_a,t_a,f_\alpha$, and $t_\alpha$; see Figure~\ref{FigQuadrilater}. Applying the triangle inequality twice, we get
\begin{align*}
d(t_\alpha, t_a)
&\leq d(t_\alpha, f_\alpha)+d(f_\alpha, f_a)+ d(f_a, t_a)\\
&\leq \bar{\varepsilon}+(1-\alpha)+\bar{\varepsilon}
=1-\alpha+2\bar{\varepsilon}.
\end{align*}
On the other hand,
\begin{align*}
1-\alpha=d(f_\alpha, f_a)
&\leq d(f_\alpha, t_\alpha)+d(t_\alpha, t_a)+ d(t_a, f_a)\\
&\leq \bar{\varepsilon}+d(t_\alpha, t_a)+\bar{\varepsilon}
=d(t_\alpha, t_a)+2\bar{\varepsilon}.
\end{align*}
Combining both inequalities,
$$(1-\alpha)-2\bar{\varepsilon}\leq d(t_\alpha, t_a)\leq (1-\alpha)+2\bar{\varepsilon}.$$
That is,
$$\alpha+2\bar{\varepsilon}\geq \mu(t_\alpha, t_a) \geq \alpha-2\bar{\varepsilon}.$$
Since $\mu(t_\alpha, t_a)=\Pr(t_\alpha \approx t_a  \mid \A)$, we obtain
$$\Pr(t_\alpha \approx t_a  \mid \A) \in [\alpha-2\bar{\varepsilon}, \alpha+2\bar{\varepsilon}].$$

It remains to remove the assumption that $\alpha$ is $n$-adic. This can be done directly, since $n$-adic numbers are dense, and for any real number we can find an $n$-adic number arbitrarily close to it. More precisely, given $\alpha' \in [0,1]$, we can find an $n$-adic number $\alpha$ such that $|\alpha-\alpha'|\leq \bar{\varepsilon}$, and then the (not necessarily $n$-adic) number $\alpha'$ also belongs to the interval $[\alpha -2\bar{\varepsilon}, \alpha +2\bar{\varepsilon}]$.
\end{proof}

\begin{figure}[tb] 
\begin{tikzpicture}[scale=0.43]
  \tikzstyle{vertex}=[circle, fill=black,minimum size=4pt,inner sep=1pt]
  \node[vertex] (A) at (2,6) {};
  \node[vertex] (B) at (0,0)   {};
  \node[vertex] (C) at (8,1)   {};
  \node[vertex] (D) at (7,7)   {};
  \draw[thick] (A) -- (B) -- (C) -- (D)--(A);
    \draw[thick] (A) -- (C);
        \draw[thick] (B) -- (D);
  \node[align=left] at (1,6) {$t_\alpha$};
  \node[align=left] at (-1,0) {$f_\alpha$};
  \node[align=left] at (9,1) {$f_a$};
  \node[align=left] at (8,7) {$t_a$};
  \node[align=left] at (-2,3) {$d(t_\alpha,f_\alpha)\leq \bar{\varepsilon}$};
  \node[align=left] at (5,-0.75) {$d(f_\alpha,f_a)=1-\alpha$};
  \node[align=left] at (10,4) {$d(f_a, t_a)\leq \bar{\varepsilon}$};
 \end{tikzpicture}
\caption{Quadrilateral in the proof of Lemma~\ref{LemmaQuadrilater}. The segments represent the distances according to the normalized Hamming metric.} \label{FigQuadrilater}
\end{figure}
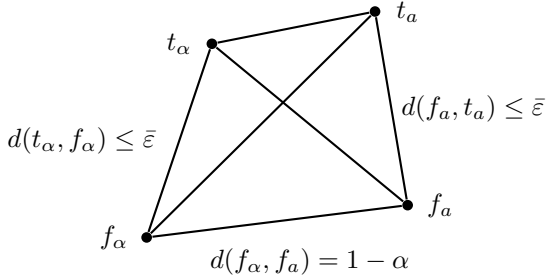

\begin{theorem}\label{TeoQuadri}
For any non-trivial algebra $\A$,
$$|\PSpec(\A)|\,\geq \,\sup_{k\geq 1} \, \left \lfloor \frac{ 1}{4(1-\Prim_k(\A))}\right\rfloor .$$
In particular,
$$|\PSpec(\A)|\,\geq \left\lfloor \frac{ 1}{4(1-\Prim (\A))}\right\rfloor.$$
\end{theorem}

\begin{proof}
Fix an arity $k>0$. If $m$ is the integer
$m=\left\lfloor \frac{1}{4\bar{\varepsilon}}\right\rfloor$, where $\bar{\varepsilon}=1-\Prim_k(\A)$, then we can construct $m$ pairwise disjoint intervals inside $[0,1]$,
$$\bigcup_{i=0}^{m-1} (4\bar{\varepsilon} i, 4\bar{\varepsilon} (i+1) ) \subseteq [0,1].$$
By Lemma~\ref{LemmaQuadrilater} above, each of these intervals contains at least one value of the spectrum of $\A$. Since this holds for every arity $k$, we can take the supremum.
\end{proof}

Since all the intervals $(4\bar{\varepsilon} i, 4\bar{\varepsilon} (i+1))$ in the previous proof have the same width $4\bar{\varepsilon}$, this theorem tells us that these $m$ spectral values are roughly uniformly distributed in the unit interval. For a fixed arity, the distance between two consecutive values can never exceed $8\bar{\varepsilon}$ in this construction. In fact, the following result holds.

\begin{corollary} \label{CoroQuadri}
If $\A$ is a non-trivial algebra such that $\sup_{k\geq 1} \Prim_k (\A)=1$, then its spectrum is dense.
\end{corollary}

\begin{proof}
Let $\varepsilon_k=\Prim_k(\A)$ and $\bar{\varepsilon}_k=1-\varepsilon_k$. As just discussed, the distance between two consecutive values in the interval decomposition for arity $k$ can never exceed $8\bar{\varepsilon}_k$. Since the supremum of the arity-$k$ primalities is 1, there exists  a subsequence $\varepsilon_{k_i}$ such that $\lim_{i\to \infty} \bar{\varepsilon}_{k_i}=0$. That is, by choosing $k$ sufficiently large, we find spectral values arbitrarily close to any given point.
\end{proof}

\begin{example}
Consider the two-element algebra
$\mathcal{A}=(\{0,1\}, \vee, \wedge)$.
It is well known that the terms of $\A$ are exactly the functions $t$ that fix $0$, that is, $t(0,\ldots,0)=0$.
For any function $f$, there is always a term $t$ that disagrees with $f$ at most at the point $(0,\ldots,0)$. Therefore, $\mu(f,t)\geq 1-1/2^k$, and hence $\lim_{k\to \infty}\Prim_k (\A)=1$. By Corollary~\ref{CoroQuadri}, the spectrum of $\A$ is dense.
\end{example}

\section{A barrier in the approximation of Boolean functions} \label{Barrera}

To conclude, we prove a general theorem about the approximation of Boolean functions.
Recall Theorem~\ref{CotaFeble}. If we apply it to obtain a bound on the primality of non-primal algebras with two elements, we obtain $\Prim(\A) \leq 3/4$. This value can be improved. We will see that if an algebra has two elements and is not primal, then its primality cannot exceed $1/2$, and that this bound is tight.

From E. L. Post’s classification of Boolean clones, \cite{Post1940}, we know that if an algebra is not primal, then the clone of its term functions is contained in one of the following five maximal clones: $\mathbf{P}_0$, the set of functions fixing $0$; $\mathbf{P}_1$, the set fixing $1$; $\mathbf{M}$, the set of monotone functions; $\mathbf{D}$, the set of self-dual functions; and $\mathbf{A}$, the set of affine functions. Each of these clones can be generated by some algebra with a finite signature.

Let us assume that $\A$ is a non-primal algebra and examine the degree of primality according to the maximal clone containing $\Clo(\A)$. Except for the affine clone, it suffices to consider unary terms to see that primality cannot exceed $1/2$, since $\Prim(\A) \leq \Prim_1(\A)$. There are four possible unary functions: the identity $\id$, the constant $1$, the constant $0$, and the self-dual function $d$, which swaps $0$ and $1$.

\begin{enumerate}
\item Suppose that $\Clo(\A) \subseteq \mathbf{P}_0$. The terms fixing $0$ are the functions $\id$ and $0$. Considering the function $1$, we have $\mu(1,\id)= 1/2$ and $\mu(1,0)=0$. Taking the function $d$, we have $\mu(d,\id)= 0$ and $\mu(d,0)=1/2$. Therefore,
$$ \Prim_1(\A)=\min_{f\in \{0,1, \id, d\}}\max_{t\in \{0,\id\}} \mu(f,t)=\frac{1}{2}.$$

\item $\Clo(\A) \subseteq \mathbf{P}_1$. This case is similar to (1), but now the unary clone contains only the functions $\id$ and $1$, and again $\Prim_1(\A)=1/2$.

\item $\Clo(\A) \subseteq \mathbf{M}$. The unary terms preserving order can only be $0$, $1$, and $\id$. Taking the self-dual function, we have $\mu(d,0)=\mu(d,1)=1/2$ and $\mu(d,\id)=0$. Therefore, $ \Prim_1(\A)=1/2$.

\item $\Clo(\A) \subseteq \mathbf{D}$. Its unary clone contains only $\id$ and $d$. Then $\mu(1,\id)=\mu(1,d)=1/2$. Hence $\Prim_1(\A)=1/2$.

\item $\Clo(\A) \subseteq \mathbf{A}$. This is the most interesting case, already studied in Example~\ref{ClonAfi}. Thanks to coding theory and bent functions, we know that the arity-$k$ primalities form a sequence satisfying
$$\frac{1}{2}+\frac{1}{2^{\frac{k}{2}+1}}\leq \Prim_k(\mathbb{Z}_2^+)\leq \frac{1}{2}+\frac{1}{2^{\frac{k+1}{2}}},$$
valid for both even and odd arity, where we recall that $\Clo(\mathbb{Z}_2^+)=\mathbf{A}$.
\end{enumerate}
Thus, $\Prim(\A)=\inf_{k\geq 1} \Prim_k(\A)=1/2$. Moreover, the affine case shows that the bound is tight. We have proved the following theorem.

\begin{theorem}
Let $\A$ be an algebra of order two. Then, $\A$ is not primal if and only if $\Prim (\A)\leq \frac{1}{2}$.
\end{theorem}

This result admits a natural interpretation: if a two-element algebra is not primal, then there are functions that the algebra can approximate in no more than half (almost half) of the points. If we want the algebra to provide better approximations, one must require full primality. In other words, there is a structural gap between $1/2$ and $1$ with respect to primality.

Regarding this last result, it is natural to ask what happens for algebras of larger cardinality. We do not currently have a clear conjecture. Investigating a generalization would likely require an analysis of the maximal clones in Rosenberg’s classification; see \cite{Rosenberg1970, Szendrei2024, Lau2006}. One may also study bounds not for global primality, but for primalities restricted to a fixed arity $k$. We leave this direction for future work.

Finally, there are two spectra of very elementary structures that remain unknown. On the one hand, the computation of $\PSpec(\mathfrak{S}_3)$ remains open. Example~\ref{ExS3} suggests that this problem is non-trivial. On the other hand, we also do not know $\PSpec((\{0,1\}, \rightarrow))$, where $\rightarrow$ denotes material implication; see again Table~\ref{TaulaOrdre2}. Although the explicit computation of these spectra would probably not substantially extend the theory developed here, these cases remain open.






\end{document}